\documentclass[12pt]{amsart}
\usepackage{graphicx,epsfig}
\usepackage{bbding,amssymb,amsfonts,amsmath,amsthm,dsfont,wasysym,pifont,stmaryrd}
\usepackage{epstopdf,yfonts}
\usepackage{hyperref}

\usepackage{color}
\usepackage{mathrsfs}
\usepackage{enumerate}

\usepackage{url}
\usepackage{amsrefs}
\usepackage{tikz}

\newtheorem{theorem}{Theorem}[section]
\newtheorem*{theorem'}{Theorem }

\newtheorem{corollary}[theorem]{Corollary}

\newtheorem{proposition}[theorem]{Proposition}

\newtheorem{cor}[theorem]{Corollary}
\newtheorem{remark}[theorem]{Remark}

\newtheorem{lemma}[theorem]{Lemma}
\newtheorem{definition}[theorem]{Definition}

\newcommand{\Nr}{\mathcal N}

\renewcommand{\(}{\left(}
\renewcommand{\)}{\right)}

\renewcommand{\^}[1]{\hat{#1}}
\renewcommand{\~}[1]{\overline{#1}}

\renewcommand{\geq}{\geqslant}
\renewcommand{\leq}{\leqslant}

\newcommand{\<}{\left\langle}
\renewcommand{\>}{\right\rangle}
\newcommand{\8}{\infty}

\renewcommand{\a}{\alpha}
\newcommand{\Aut}{\text{Aut}}

\renewcommand{\Cap}[1]{\underset{#1}{\cap }}
\newcommand{\ch}[1]{\check{#1}}

\newcommand{\eye}{\sphericalangle}

\newcommand{\f}{\varphi}

\newcommand{\fix}{\textnormal{Fix}}
\newcommand{\frakH}{\mathfrak{H}}
\newcommand{\g}{\gamma}
\newcommand{\G}{\Gamma}

\renewcommand{\int}{\varint}

\newcommand{\Lim}[1]{\underset{#1}{\lim}}

\newcommand{\N}{\mathbb{N}}

\newcommand{\Out}{\text{Out}}

\newcommand{\Ps}{{\mathbb{P}}}

\newcommand{\R}{\mathbb{R}}

\newcommand{\Z}{\mathbb{Z}}

\DeclareMathOperator{\Prob}{Prob}
\DeclareMathOperator{\Supp}{Supp}

\newcommand{\bdr}{\partial_{reg}}

\DeclareMathOperator{\Id}{Id}

\newcommand{\CX}{{\mathscr{C}\!X}}

\newcommand{\eps}{\varepsilon}

\title[Contact Graphs, Boundaries, and a CLT]{Contact Graphs, Boundaries, and a Central Limit Theorem for CAT(0) cubical complexes}
\author{Talia Fern\'os, Jean L\'ecureux, Fr\'ed\'eric Math\'eus}
\thanks{\\
J. L\'ecureux was partially supported by Projet ANR-14-CE25-0004 GAMME and ANR-16-
CE40-0022-01 AGIRA\\ T. Fern\'os was partially supported by NSF DMS Grant \#2005640}

\begin{document}
\maketitle

\begin{center}

\textit{In memory of \'Emile Le Page}
\end{center}

\begin{abstract}
Let $X$ be a nonelementary CAT(0) cubical complex. We prove that if $X$ is essential and irreducible,  then the contact graph of $X$ (introduced in \cite{Hagen}) is unbounded and its boundary is homeomorphic to the regular boundary of $X$ (defined in \cite{Fernos}, \cite{KarSageev}). Using this, we  reformulate the Caprace-Sageev's Rank-Rigidity Theorem in terms of the action on the contact graph.
Let $G$ be  a group with a nonelementary action on $X$, and $(Z_n)$ a random walk corresponding to a generating probability measure on $G$ with finite second moment. Using this identification of the boundary of the contact graph, we prove a Central Limit Theorem for $(Z_n)$, namely that
$\frac{d(Z_n o,o)-nA}{\sqrt n}$ converges in law to a non-degenerate Gaussian distribution (where $A=\lim \frac{d(Z_no,o)}{n}$ is the drift of the random walk, and $o\in X$ is an arbitrary basepoint).
\end{abstract}

\tableofcontents

\section{Introduction}

Let $X$ be a finite-dimensional CAT(0) cubical complex, $G$ a discrete countable group acting by cubical automorphisms on $X$. Let $\mu$ be an admissible probability measure on $G$, i.e. a measure such that the semigroup generated by $\Supp(\mu)$ is $G$. We consider the associated random walk. Namely, let $(g_i)_{i\geq 1}$ be a sequence of independant, identically distributed random variables of law $\mu$ and $Z_n=g_1\dots g_n$.

Fix an origin $o\in X$. In our previous paper \cite{FLM}, we studied the behavior of the sequence $\{Z_n.o\}$. We proved in particular that if the action of $G$ on $X$ is nonelementary then $(Z_n.o)$ almost surely converges to some regular point  in the Roller boundary $\eta_{(Z_n)}\in\partial_{reg} X$. 
Furthermore there is a $\lambda>0$ such that $\Lim{n\to\8}\frac{d(Z_no,o)}{n}=\lambda$, where $d$ is the combinatorial distance on $X$.

In this paper, we prove the following Central Limit Theorem:

\begin{theorem}\label{maintheorem}
Assume that the action of $G$ on $X$ is nonelementary.
Assume that $\mu$ has  finite second moment, that is,
$$\int_G (d(g o,o))^2 \mathrm d\mu(g)<+\infty.$$

Then there exists $\sigma>0$ such that $\frac{d(Z_no,o)-n\lambda}{\sqrt n}$ converges in distribution to a centered Gaussian law of variance $\sigma^2$.
\end{theorem}

This statement has several predecessors in various contexts, depending on more or less stringent moment conditions. If we restrict ourselves to Central Limit Theorems for random walks on non-abelian groups, the main contributors are:
\begin{itemize}
\item \'Emile ~Le Page, for random walks on linear groups, under an exponential moment condition {\it i.e.} 
$\int_G \exp(\alpha d(g o,o)) \mathrm d\mu(g)<+\infty$ for some positive $\alpha$ \cite{LePage_82}.
\item Stanley ~Sawyer and Tim ~Steger, in the case of a free product of several copies of $\Z/2\Z$ if $\mu$ has first 
moment of order greater than 4 \cite{Sa_Ste_87}.
\item Fran\c cois ~Ledrappier, who adapted Le Page's strategy to random walks on finitely generated free non-abelian groups under an exponential moment condition \cite{Led01}.
\item Michael ~Bj\" orklund, who extended Ledrappier's result to a certain class of hyperbolic groups (still under an exponential moment condition)
\cite{Bjorklund}.
\item Yves ~Benoist and Jean-Fran\c cois ~Quint, who proved a CLT for random walks on linear groups \cite{BQLinear} and on any hyperbolic group 
\cite{BenoistQuintHyp}, in both cases assuming only that $\mu$ has finite moment of order 2.
\item Camille ~Horbez \cite{Horbez}, adapting the proof of Benoist and Quint for mapping class groups and $\Out(F_n)$.
\item Pierre ~Mathieu and Alessandro ~Sisto \cite{MathieuSisto} prove the Central Limit Theorem for acylindrically hyperbolic groups, assuming again an exponential moment.
\item Ilya ~Gekhtman, Samuel J. ~Taylor, and Giulio ~Tiozzo \cite{gekhtman2020central} prove the Central Limit Theorem for the counting measure or for random walks on groups with an automatic structure acting on hyperbolic spaces, and apply this to some nonpositively curved groups.
\end{itemize}

Indeed, Le Page - followed by Ledrappier and Bj\" orklund - proved a stronger result, namely a spectral gap property
for some Markov operator $Q$, and used it to derive a Central Limit Theorem. It is the reason why an exponential moment 
condition on $\mu$ is needed. Using the invertibility of $I-Q$, they proved the existence of a martingale $M_n$ such that
$M_n- d(Z_n o,o)+n\lambda$ is bounded, and applied the Central Limit Theorem for martingales to get the conclusion.
Benoist and Quint do not prove such a spectral gap property. They give an explicit formula for the martingale $M_n$,
by means of some integral that is defined on a suitable boundary. They prove the existence of this integral under only a finite
second moment condition for $\mu$. The proof of \cite{MathieuSisto} for acylindrically hyperbolic groups follows a different strategy: they obtain the Central Limit Theorem from large deviation inequalities. We note that there is some intersection between our results and theirs, as some, but not all, groups acting on CAT(0) cubical complexes are acylindrically hyperbolic (see for example \cite{ChatterjiMartin}). In the case of \cite{gekhtman2020central}, there are technically no overlaps in results since we are examining the (orbit) random walk on the CAT(0) cubical complex, whereas they are considering the random walk on a hyperbolic space (which would be the contact graph in this case). Of course, they also assume the groups are automatic. 

CAT(0) cubical complexes have a functorial construction \cite{Sageev_95, Roller, Chatterji_Niblo, Nica} which makes them appear, together with groups acting on them, quite naturally and often ``in nature". The list of examples is long, and is sure to continue to grow. We mention just a few: the class of right angled Artin groups \cite{CharneyDavis}, Coxeter groups \cite{NibloReeves-Coxeter}, the Higman group with presentation $\< a_i, i\in \Z/4\Z\mid a_{i+1}^2=a_ia_{i+1}a_i^{-1}\>$  \cite{Martin}, random groups of density less than $\frac{5}{24}$\cite{OllivierWise, MackayPrzytycki}, in the square model at density less than $\frac{1}{3}$ \cite{Duong}, fundamental groups of closed hyperbolic 3-manifolds \cite{Agol,KahnMarkovic,WiseMalnormal}, and small cancelation groups \cite{WiseSmall}. 

Among the group actions appearing in this list all are proper, almost all are co-compact, and most  are nonelementary. As we only require that the action is nonelementary, we get an even larger family of groups. For example, our results apply to the nonelementary Bestvina-Brady kernels, such as the original one introduced by Stallings, namely, the kernel of the map $F_2\times F_2\to \Z$ where the standard generators are mapped to $1\in \Z$  \cite{Stallings, BestvinaBrady}, which is clearly nonelementary.

\medskip

Our result adds to the limited list of CAT(0) examples that satisfy the Central Limit Theorem. While Benoist and Quint have both linear 
(hence acting on CAT(0) spaces, such as symmetric spaces)
 and hyperbolic results, it is worth noting that we follow their strategy for the hyperbolic case \cite{BenoistQuintHyp}. In order to do so in our setting, we are first led to study the random walk on the contact graph $\CX$ simultaneously with the random walks on $X$, when $X$ is essential and irreducible. To do so, we employ the celebrated work of Maher and Tiozzo that describes the behavior of random walks on hyperbolic spaces \cite{MaherTiozzo}. The contact graph, introduced by Hagen in \cite{Hagen}, is a hyperbolic graph associated to a CAT(0) cubical complex. We then must consider the case when $X$ is reducible, and in particular, when our acting group $G$ does not preserve each irreducible factor. 

Our study of the contact graph led us to better  understand the  boundary of the contact graph $\CX$. Recall that the \emph{regular boundary} $\bdr X$ is a subset of the Roller boundary and is therefore endowed with the restriction topology (see Definition \ref{def:regular pt}). We show:

\begin{theorem}
Let $X$ be a  finite-dimensional CAT(0) cubical complex. There exists an $\Aut(X)$-equivariant homeomorphism between $\bdr X$ and $\partial \CX$.
\end{theorem}

\subsection*{Acknowledgments} We thank Anthony Genevois and Mark Hagen for useful discussions concerning  the contact graph and its boundary (\S \ref{sec:comparing}). We would also like to thank the referee for his or her careful review of the paper.

\section{Background}

We recall here some of the basic facts we will need in the sequel. We refer the interested reader to \cite{CapraceSageev}, \cite{Fernos} or \cite{FLM} for further background on CAT(0) cubical complexes.   

We adopt the conventions and notations of \cite{FLM}. In particular, we identify $X$ with its set of vertices, and we endow it with the combinatorial distance $d$, given by the number of hyperplanes separating two vertices. By a geodesic in $X$, we always mean a combinatorial geodesic, i.e. a path between two vertices which is geodesic for the combinatorial distance $d$. Throughout the text, all CAT(0) cubical complexes will be assumed to be finite-dimensional. We shall often refer to a basepoint $o\in X$ which is arbitrary, unless otherwise specified in some local context.

We shall denote by $\Aut(X)$ the group of cubical automorphisms of $X$, that is the bijective self maps of $X$ that preserve the given cubical structure of $X$. 

 We denote the set of half-spaces by $\frakH$, and the set of hyperplanes by $\^\frakH$. There is a natural involution $\frakH\to \frakH$ that takes a half-space to its complement. We shall use this often and denote the assignment $h\mapsto h^*$. Similarly, the hyperplane associatedt to $h$ is denoted $\^h$ and recall that $\^h^* = \^h$. Thanks to Sageev's Duality \cite{Sageev_95}, \cite{Nica}, \cite{ChatterjiNiblo}, there is a continuous injection of $X$ into $2^{\frakH}$: a vertex $x$ corresponds to the set $U_x$ of half-spaces containing $x$.
The \emph{Roller compactification} $\overline X$ is the closure of $X$ in $2^{\frakH}$.  The \emph{Roller boundary} is $\partial X=\overline X\setminus X$.  We extend naturally the notation $U_x$ to the boundary, by defining $U_x=x$ for $x\in\partial X$.

\begin{remark}
We note that $\partial X$ is only closed when $X$ is locally compact. 
 Indeed, observe that $X$ is locally compact if and only if it is locally finite. Therefore, if $X$ is not locally compact then  there is a vertex $x_0\in X$ with infinitely many edges emanating from it. Since $X$ is finite dimensional, there are infinitely many of these whose dual hyperplanes are not transverse. Let $\{h_n\}_{n\in \N}$ be an enumeration of the associated half-spaces which do not contain $x_0$, i.e $x_0\in \Cap{n}h_n^*$. Since $X$ is essential, we have that all half-spaces are deep, and in  particular, there exists $\xi_n\in h_n\cap \partial X$. Furthermore, since $\~X$ is compact, up to passing to a subsequence, we may assume that $\xi_n\to x_1\in \~X$. We now show that $x_1=x_0$. To this end, it is sufficient to let $h$ be a half-space dual to an edge emanating from $x_0$ with $x_0\in h^*$ and prove that $x_1\in h^*$ as well. But for each $n$ we must have that $h\cap h_n=\varnothing$ or $h\pitchfork h_n$. Using finite dimensionality again, there is and $N$ so that for all $n>N$ we must have $h\cap h_n=\varnothing$. Therefore, for all $n>N$ we have that $\xi_n \in h^*$ and hence $x_1\in h^*$. This completes the proof.  
 \end{remark}
 
Sageev's duality then also extends to Sageev-Roller Duality \cite{Roller} and this also allows for the extension of the combinatorial metric: for $x,y \in \~ X$ let $d(x,y) =\frac{1}{2} \#(U_x\triangle U_y) \in [0,\8]$.

\begin{definition}
The \emph{median} of $x,y,z\in\overline X$ is the point $m(x,y,z)$ defined by the formula
$$U_{m(x,y,z)}=(U_x\cap U_y)\cup (U_y\cap U_z)\cup (U_z\cap U_x)$$
\end{definition}

Equivalently, define $I(x,y)$ as the intersection of all half-spaces containing both $x$ and $y$. Then the median point $m(x,y,z)$ is the unique point in $I(x,y)\cap I(x,z)\cap I(y,z)$.

\begin{remark}\label{Rem: MedianAxioms}
 The median satisfies some nice topological and algebraic properties. 
 \begin{itemize}
\item The median as a map $m: \~X^3 \to  \~X$ is continuous \cite[Lemma 6.21]{FLM}
\item The median is invariant under permutation of its three variables.
\item $m(x,x,y) = x$ for all $x,y \in \~X$.
\item For all $a,b,x,y,z\in \~X$ we have that 
$$m(a,b, m(x,y,z))=m(m(a,b,x), m(a,b,y), z).$$
\end{itemize}
\end{remark}

While the last 3 properties are well known for the median as a map on $X^3$, the extension of these to ${\~X}^3$ follows from continuity, or are also easily verified directly.
Also, though we will not use this, it is perhaps worth giving an idea of the meaning of the last property. If $x,y \in \~X$ then the median map yields a CAT(0) cubical morphism $\f_{(x,y)}: \~X\to I(x,y)$ given by $\f_{(x,y)}(z) := m(x,y,z)$. Therefore, the last property can be seen as a algebraic morphism property on such pairs. Namely, 

$$ \f_{(a,b)} \circ \f_{(x,y)} = \f_{ \( m(a,b,x),m(a,b,y)\) }
$$

%
%
%

\begin{definition}
Let $h,k \in \frakH$ be half-spaces, and $\^h,\^k$ their respective boundary hyperplanes. We say that $h$ and $k$ are
\begin{itemize}
\item \emph{transverse} if the four intersections $h\cap k, h^*\cap k, h\cap k^*, h^*\cap k^*$ are nonempty. In this case, we write $h\pitchfork k$;
\item \emph{parallel} if they are not transverse, or equivalently if up to possibly reversing one or both orientations they are nested, i.e. we have that $h\subseteq k$;

\item \emph{strongly separated} if they are parallel and no half-space is transverse to both;
\item \emph{tightly nested} if up to reversing orientations, $h\subsetneq k$ and no other half-space $\ell$ is properly between, i.e. if $h\subseteq \ell \subseteq k$ then $\ell =h$ or $k$. 
\item These notions do not depend on the choice of the orientation, and therefore we say that $\^h$ and $\^k$ are parallel, transverse, strongly separated or tightly nested if any choice of orientation satisfies these properties.
\end{itemize}
\end{definition}

\begin{definition}
Let $X$ be a finite dimensional CAT(0) cubical complex. The action $G\to \Aut(X)$ is said to be:
\begin{itemize}
\item \emph{nonelementary} if there is no finite $G$-orbit in $X\cup \partial_\eye X$, where $\partial_\eye X$ is the CAT(0) visual boundary of $X$;
\item \emph{Roller-nonelementary} if there is no finite $G$-orbit in the Roller compactification $\~X$;
\item \emph{essential} if for every $h\in \frakH$ there is a $g\in G$ such that $gh\subset h$.
\end{itemize}
\end{definition}

\begin{remark}\label{rem: basic action defs}
We take a moment to say more about the above definitions.
 \begin{enumerate}
\item We note that momentarily thinking of $X$ as the full CAT(0) cubical complex with its CAT(0) metric, then there is a finite $G$-orbit in $X$ if and only if there is a fixed point in $X$.
\item We shall later need the fact that if an action is nonelementary, then it is Roller nonelementary \cite[Proposition 2.26]{CFI}.
\item It is worth noting that this is not in fact Caprace and Sageev's definition of an essential action. However, it is equivalent for nonelementary actions thanks to their Double Skewering Lemma. Since we shall always make the nonelementary assumption, we provide this simpler definition.
\end{enumerate}

\end{remark}

The notion of strongly separated hyperplanes was introduced by Behrstock and Charney \cite{Behrstock_Charney}. 
In \cite{CapraceSageev} Caprace and Sageev provide several important properties which have been crucial in recent developments in the theory of CAT(0) cubical complexes, in particular, their irreducibility criterion.

\begin{theorem}[\cite{CapraceSageev}]\label{Th: CapraceSageev Fundamentals} Let $X$ be a finite dimensional CAT(0) cubical complex with $\Aut(X)$ (or equivalently the existence of a subgroup) acting essentially and nonelementarily. 
 \begin{enumerate}
\item (\emph{Irreducibility Criterion}) $X$ is irreducible if and only if there exist $h,k \in \frakH$ that are strongly separated. 
\item If $X$ is not reducible, then it admits a unique (up to permutation) irreducible decomposition $X \cong X_1\times \cdots \times X_D$ such that $\Aut (X)$ contains an isomorphic copy of $\Aut (X_1)\times \cdots \times \Aut (X_D)$ as a finite index subgroup. Permutation of factors and their cosets are the only automorphisms missing from this product decomposition. 
\end{enumerate}

\end{theorem}

\begin{definition}\label{def:regular pt}
A point $\xi \in \partial X$ is called \emph{regular} if, for every $h,k\in U_\xi$, there exists $l\in U_\xi$ such that $l\subset h\cap k$ and $l$ is strongly separated from both $h$ and $k$.

Equivalently \cite[Proposition 7.5]{Fernos}, the point $\xi$ is regular if and only if there exists $h_{n+1}\subset h_n$, an infinite decreasing chain of pairwise strongly separated half-spaces such that $\Cap{n\in \N}h_n=\{\xi\}$.
\end{definition}

\begin{remark}
In \cite{Fernos} the regular boundary of a product is defined to be the product of the regular boundaries of each irreducible factor. Here in order to make the statement of Theorem \ref{thm:bdrCX} simple we prefer to keep the same definition as above in all cases, so that the regular boundary of a product is empty.
\end{remark}

\begin{definition}
 An element $g\in \Aut (X)$ is said to be \emph{regular rank-1} if  there exists $n\in \N$ and $h\in \frakH$ so that $g^n h \subset h$ are strongly separated.
\end{definition}

\begin{definition}
The \emph{translation length} of an isometry $g$ on a metric space $\mathcal Z$ is defined by fixing $x\in \mathcal Z$ and looking at 
$$\ell(g)=\lim_n \frac{d(x,g^nx)}{n}.$$
Furthermore, we say $g$ is
\begin{itemize}
\item   \emph{loxodromic} if $\ell(g)>0$ and
\item \emph{elliptic} if some (or equivalently every) $\<g\>$-orbit is bounded, in which case $\ell(g)=0$. 
\end{itemize}

\end{definition}

\begin{remark} A few things worth noting here:
\begin{itemize}
  \item It is straightforward to verify that this definition does not depend on the choice of $x\in \mathcal Z$. This translation length is sometimes called the stable translation length to distinguish it from:
 $$\ell'(g)=\inf_{x\in \mathcal Z} d(x,gx).$$
  When $\mathcal Z=X$ is a CAT(0) cubical complex endowed with the combinatorial metric, Haglund showed that if $\ell'(g)>0$ then $g$ has a combinatorial axis on which it acts as a translation by $\ell'(g)$ and  that  $\ell'(g^n) = n\ell'(g)$ for every $n\in \N$. \cite[Corollary 6.2]{Haglund}. Conversely, if $\ell'(g)=0$ then there is a $\<g\>$-fixed vertex \cite[Theorem 6.3]{Haglund}\footnote{Actually, this is not quite correct: an automorphism may not act ``stably without inversions". However, up to passing to the the cubical subdivision of $X$ the result is correct as stated and $\ell(g) = \ell'(g)$. Furthermore, passing to the cubical subdivision will not impact any of our results in this paper.}. 
\item For $X$ a finite dimensional CAT(0) cubical complex, an element $g\in \Aut (X)$ is loxodromic  if and only if there exists $h\in \frakH$ and $n\in\Z$ so that $g^nh \subset h$ (see for example \cite[Remark 3.2]{FFT}). 
\end{itemize}
\end{remark}

Recall that in a metric space $X$ the \emph{Gromov product} is defined, for every $x,y\in X$, as $(x|y)_o=\frac{1}{2}(d(x,o)+d(y,o)-d(x,y))$. In a CAT(0) cubical complex it is easy to check that this coincides with $d(o, m(x,y,o))$, where $m(x,y,o)$ is the median of $x,y$ and $o$. This allows us to extend the Gromov product to the boundary:

\begin{definition}
The \emph{Gromov product} based at $o\in X$ between $x,y\in\overline X$ is defined by
$$(x|y)_o = d(o, m)$$
where $m$ is the median of $x,y$ and $o$. (If $m\in \partial X$ then $d(o,m)= \8$).
\end{definition}

 By continuity of the median  the Gromov product is a continuous function on $\overline X\times\overline X$.

\begin{definition}
Fix $o\in X$. For each $x\in \overline X$ define the \emph{horofunction} based at $o$ with respect to $x$ to be the map $h_x: X\to \R$ as $h_x(a) = d(o,a)-2(a|x)_o$. 
\end{definition}

\begin{remark}

Note that if $x\in X$ then $h_x(a) = d(x,a)-d(x,o)$. Furthermore, the map $(x,a)\mapsto h_x(a)$ is continuous on $\overline X\times X$. It follows that if $(x_n)_{n\in\N}\in X^\N$ is a sequence with $\lim x_n=x$, then the horofunction associated to $x$ is the pointwise limit of the functions $h_{x_n}(y):= d(x_n,y)-d(x_n,o)$, which is the more usual definition of a horofunction.


\end{remark}

\begin{remark}\label{rem:Gproduct} 
For every $a,z\in X$, $x,y\in\overline X$, by applying the definitions and taking limits where necessary we have:
\begin{itemize}
 \item[]$h_x(g^{-1}o) = -2(x|g^{-1}y)_o+2(gx|y)_o+h_y(go)$
 \item[] $h_x(a) = d(a,m(a,x,o)) - d(o,m(a,x,o))$
  \item[] $2(a|x)_z = d(a,z)+h_x(z)-h_x(a)$.
\end{itemize}

\end{remark}

\begin{definition}
Let $G$ be a group acting by permutations on the set $Z$, and $E$ an abelian group. A map $\sigma:G\times Z\to E$ is said to be a \emph{cocycle} if 
 $$\sigma(gg',x)=\sigma(g,g'x)+\sigma(g',x).$$
\end{definition}

\begin{remark}\label{rem: horo is cocycle}
 By an elementary calculation (see \cite[Lemma 6.22]{FLM}), the map $\sigma:G\times \overline X\to \mathbb R$ defined by $\sigma(g,x)=h_x(g^{-1}o)$ is a continuous cocycle.
\end{remark}

 \section{Contact graph} 
 
 \subsection{Basic definitions and properties}
 
 The contact graph was introduced by Hagen \cite{Hagen}. Since then it has proved quite useful in the study of nonpositively curved spaces. For example, it was fundamental in the development of hierarchically hyperbolic spaces and groups \cite{BHS1,BHS2}. We first recall a few basic facts about this graph. 
 
 \begin{definition}
 The \emph{contact graph} of $X$, denoted by $\CX$, is the graph whose vertices are hyperplanes of $X$, and an edge connects two hyperplanes if either they are tightly nested or transverse.
 \end{definition}

 We denote $d_{\CX}$ the distance function in the contact graph.
 
 \begin{remark}\label{rem:3->ss}
 If $\^h$ and $\^k$ are two hyperplanes such that $d_{\CX}(\^h,\^k)\geq 3$ then $\^h$ and $\^k$ are strongly separated. 
 \end{remark}
 
Indeed if $h$ and $k$ are distinct, parallel, and not strongly separated then there exists a hyperplane $\ell$ transverse to both, meaning that $d_{\CX}(\^h, \ell)=1$ and $d_{\CX}(\^k, \ell)=1$ so that $d_{\CX}(\^h, \^k)\leq2$.

 The following is the key fact that we shall need about the contact graph.
 
 \begin{theorem}[\cite{Hagen}]
 The contact graph $\CX$ is quasi-isometric to a tree and in particular, it is hyperbolic.
 \end{theorem}

Recall that a \emph{clique} of a graph is a complete subgraph. It is possible to define a type of projection from $X$ to cliques in $\CX$, as follows: 

To a vertex $x\in X$, one can associate the set $\pi(x)$ of hyperplanes which are adjacent to $x$. By definition, any two hyperplanes adjacent to $x$ are adjacent in the contact graph, so that $\pi(x)$ is a clique of $\CX$. This defines an $\Aut(X)$-equivariant map from $X$ to the set of cliques of $\CX$.

We define the distance $d_{\CX}(C,C')$ between two cliques $C$ and $C'$ as the minimum of $d_{\CX}(x,x')$ where $x\in C$ and $x\in C'$.

If $\gamma$ is a (combinatorial) geodesic in $X$, one can define $\pi(\gamma)$ as the path in $\CX$ formed by all hyperplanes crossed by $\gamma$.  In general it may or may not be a geodesic in $\CX$. But it is always a path starting at a vertex in $\pi(x)$ and ending at a vertex in $\pi(y)$. We note that we have defined two maps with the same symbol $\pi$ with different domains; one is $X$ while the other is ``geoedesics in $X$". We hope that this slight abuse of notation will not cause any confusion.   We deduce easily, using Remark \ref{rem:3->ss}, the following Lemma:

\begin{lemma}\label{lem:projections}
Let $A\geq 3$. For any $x, y\in X$ we have $d_{\CX}(\pi(x),\pi(y))\leq d(x,y)$. Furthermore, let $\g$ be a  combinatorial geodesic segment  $\g$  connecting $x$ and $y$. 
If
   $d_{\CX}(\pi(x),\pi(y))\geq A$ then $\gamma$ crosses at least $\left\lfloor A/3\right\rfloor$-many pairwise strongly separated hyperplanes. 
\end{lemma}

\begin{proof}
Since $\pi(\g)$ is a path from $\pi(x)$ to $\pi(y)$ of length $d(x,y)$ we have $d_{\CX}(\pi(x),\pi(y))\leq d(x,y)$. If $d_{\CX}(\pi(x),\pi(y))\geq A$ then $\pi(\g)$ is of length at least $A$, so there are at least $\left\lfloor\dfrac{A}{3}\right\rfloor$-many vertices of $\pi(\g)$ which are at pairwise distance at least $3$; by  Remark \ref{rem:3->ss} they are strongly separated. 
\end{proof}

Recall that if $\^h$ is a hyperplane, then $\Nr(\^h)$ denotes the set of all vertices of $X$ which are adjacent to $\^h$ and is a CAT(0) cubical complex in its own right. If $h$ is a half-space, we also denote $\Nr(h)=\Nr(\^h)$.  

We will make extensive use of the construction given by \cite[Proposition 3.1]{BHS1}, which states the following.

\begin{proposition}\label{prop:hierarchy}
Let $x,y\in X$. Then there exists a combinatorial geodesic $\gamma$ from $x$ to $y$, which is obtained as a concatenation of geodesics $\gamma_0\gamma_1\cdots \gamma_K$ such that there exist hyperplanes $\^h_i$ with $\gamma_i\subset  \Nr(\^h_i)$, and the sequence $(\^h_0,\dots,\^h_K)$ is a geodesic in $\CX$.
\end{proposition}

Such a geodesic $\gamma$ is called a \emph{hierarchy path} between $x$ and $y$. The sequence $(\^h_0,\dots,\^h_K)$ is called  \emph{the projection} of $\gamma$ on $\CX$, despite the fact that it may not be unique. Furthermore, if one of the edges in $\g_i$ is dual to $\^h_j$ (that is, $\g_i$ intersects both $h_j$ and $h_j^*$) then we say that $\g_i$ \emph{crosses $\^h_j$}.

\subsection{Comparing boundaries} \label{sec:comparing}

In this section, our goal is to show there is a cannonical  homeomorphism between the boundary of the contact graph and the regular boundary $\bdr X$

We begin with briefly recalling the definition of the boundary of a non-proper hyperbolic space $\mathcal Z$ (see for example \cite{Nonproper}). Choose a basepoint $o\in \mathcal Z$. A \emph{Gromov sequence} is  a sequence  $(x_n)_{n\in\N}$ of points of $\mathcal Z$ satisfying that $(x_n|x_m)_o$ tends to $+\infty$ as $n,m \to +\8$. Two Gromov sequences $(x_n)_{n\in\N}$ and $(y_n)_{n\in\N}$ are equivalent if $(x_n|y_m)_o$ also tend to $+\infty$. The boundary $\partial \mathcal Z$ is the set of all equivalence classes of Gromov sequences, and we say that the Gromov sequence $(x_n)_{n\in\N}$ converges to its equivalence class. This convergence does not depend on the choice of the basepoint $o$.

Let $\xi,\eta\in \mathcal Z\cup  \partial \mathcal Z$. The Gromov product extends to the boundary by taking $(\xi |\eta)_o=\inf\{ \liminf_{n,m\to +\infty} (x_i|y_j)_o\}$, where the infimum is taken over all classes of Gromov sequences $(x_n)_{n\in \N}$ and $(y_m)_{m\in\N}$ converging to $\xi$  and $\eta$ respectively (and $\liminf_{n,m} (x_n|y_m)_o=\sup\limits_{(n,m)\in \N^2} \left(\inf\limits_{i\geq n, j\geq m}  (x_i|y_j)_o\right)$).
The topology on $\mathcal Z\cup \partial \mathcal Z$ (which is metrizable) can be defined by saying that a sequence $\xi_n$ converges to $\xi$ if $(\xi_n|\xi)_o$ converges to $+\infty$.

\begin{lemma}\label{lem:hierarchyGromov}
 Suppose $h_0\subset \cdots \subset h_N$ are distinct pairwise strongly separated half-spaces. Then for every $0\leq n,m\leq N$ we have (in $\CX$) $(\^h_n,\^h_m)_{\^h_0}\geq \min(n,m)-3$.
 
 Furthermore for every $0<k<N$ we have $(\^h_0|\^h_N)_{\^h_k}\leq 3$.
\end{lemma}


\begin{proof}
Let $x\in h_0$ and $y\in h_N^*$, and let $\gamma$ be a hierarchy path from $x$ to $y$ with  geodesic projection $(\^c_1,\dots,\^c_M)$ as in Proposition \ref{prop:hierarchy}; meaning that $\gamma$ is a concatenation of geodesics $\gamma_1,\dots,\gamma_M$ with $\gamma_i\subset \Nr(\^c_i)$, and $\^c_1,\dots,\^c_M$ is geodesic in $\CX$. For every $i$, the geodesic $\gamma$ must cross $\^h_i$ at some point, so there exists $j(i)$ such that $\gamma_{j(i)}$ crosses $\^h_i$. 
 It follows that $\^c_{j(i)}$ is either transverse or equal to $\^h_i$, and hence  $d_{\CX}(\^c_{j(i)},\^h_i)\leq 1$. 
  Note also that since $\^h_{i+1}$ must be crossed by $\gamma$ after $\^h_i$ we have that $j(i+1)\geq j(i)$.
Furthermore, let $i\neq k$. Since $h_i$ is strongly separated from $h_k$, and $\^c_{j(i)}$ is transverse or equal to $\^h_i$, it follows that $\^c_{j(i)}$ is not transverse (nor equal) to $\^h_k$, hence $j(i)\neq j(k)$. Therefore $j$ is strictly increasing.
 
Now assume that $0\leq n\leq m\leq N$. Since $d_{\CX}(\^h_0, \^c_{j(0)})\leq 1$ we have $d_{\CX}(\^h_0,\^h_n)\geq d_{\CX}(\^c_{j(0)},\^c_{j(n)})-2 = j(n)-j(0)-2$. Similarly, $d_{\CX}(\^h_0,\^h_m)\geq j(m)-j(0)-2$. Therefore $$d_\CX(\^h_n,\^h_m)\leq d_\CX(\^c_{j(n)},\^c_{j(m)})+2=j(m)-j(n)+2.$$ Summing these three inequalities, we get $2(\^h_n|\^h_m)_{\^h_0}\geq 2j(n)-6$. Now since $j:\N\to \N$ is strictly increasing we have $j(n)\geq n$, hence $(\^h_n|\^h_m)_{\^h_0}\geq n-3. $ Similarly if $m\leq n$ we have $(\^h_n|\^h_m)_{\^h_0}\geq m-3 $, proving the first inequality. 

 Now to prove the other inequality, we fix $0\leq k\leq N$. Using the same argument as above we have $d_{\CX}(\^h_0,\^h_k)\leq d(\^c_{j(0)},\^c_{j(k)})+2=j(k)-j(0)+2$ and $d_{\CX} (\^h_N,\^h_k)\leq j(N)-j(k)+2$ as well as $d_{\CX} (\^h_N,\^h_0)\geq j(N)-j(0)-2$. Summing these three inequalities leads to $2(\^h_0|\^h_N)_{\^h_k}\leq 6$.

\end{proof}

We are now ready to state and prove the main result of this section:

\begin{theorem}\label{thm:bdrCX}
Let $X$ be a finite-dimensional CAT(0) cubical complex. There exists an $\Aut(X)$-equivariant homeomorphism $ \bdr X\to \partial \CX$.
\end{theorem}

\begin{proof}

\medskip

\textbf{Definition of $\f:\bdr X\to \partial \CX$. }
Let us first define  the map $\f:\bdr X\to \partial \CX$. If $\bdr X$ is empty, there is nothing to do. If not, let $\alpha\in \bdr X$. By \cite[Proposition 7.4]{Fernos}, there exists a decreasing sequence $(h_n)$ of pairwise strongly separated half-spaces such that $\bigcap_{n}h_n=\{\alpha\}$. By Lemma \ref{lem:hierarchyGromov} we have $\lim\limits_{n,m\to +\infty}(\^h_m,\^h_n)_{\^h_0}=+\infty$. Therefore $(\^h_n)_{n\in \N}$ is a Gromov sequence, and thus defines a point $\f(\alpha)$ (and in particular $\partial \CX\neq \varnothing$).

We claim that this point $\f(\alpha)$ does not depend on the sequence $(h_n)$. To this end, fix a base point $o\in X$ and let $(k_n)$ be another sequence with $\bigcap_{n}k_n=\{\alpha\}$. For all $m$ large enough, $k_m$ separates $o$ from $\alpha$. Fix such an $m$.  Then there is an $N_m>m$ such that for every $n>0$, we get that $h_{n+N_m}\subset k_m$, are strongly separated \cite[Lemma 5.11]{FLM}. Then the sequence $(h'_n)$ defined by $h'_n=k_n$ if $n\leq m$ and $h'_n=h_{n+N_m}$ if $n>m$  is a decreasing sequence of pairwise strongly separated half-spaces. Hence by Lemma \ref{lem:hierarchyGromov} we have for $n>m$, $(\^k_m|\^h_{n+N_m})_{\^k_0}=(\^h'_m|\^h'_n)_{\^h'_0}\geq m-3$ which tends to $+\infty $ as $m$ tends to $+\infty$.

\medskip

\textbf{Continuity of $\f$.} Suppose $(\alpha_n)_{n\in \N}$ is a sequence of points in $\bdr X$ with $\alpha_n \to \alpha$ and $(h_m)_{m\in \N}$ a sequence of strongly separated half-spaces containing $\alpha$. By the previous argument the sequence $(\^h_m)_{m\in \N}$ represents $\phi(\alpha)$ in $\CX$, so that it suffices to prove that $(\varphi(\alpha_n)|\^h_m)_{\^h_0}$ tends to $+\infty$ as $n$ and $m$ tend jointly to $+\infty$.  Fix  $m>0$. Since $\alpha_n$ converges to $\alpha$ there exists $N$ such that for $n>N$ we have $\alpha_n \in h_i$ for every $i\leq m$. Using  \cite[Lemma 5.11]{FLM} again we see that for every $n$ there is a decreasing sequence $(k_j)_{j\in \N}$ of strongly separated half-spaces  containing $\alpha_n$ such that $k_j=h_j$ for $j\leq m$. Let $j>m$; using Lemma \ref{lem:hierarchyGromov}, we get again that $(\^h_m|\^k_j)_{\^h_0}\geq m-3$. Letting $j$ tend to $+\infty$, it follows that for every $n>N$ we have $(\^h_m|\varphi(\alpha_n))_{\^h_0}\geq m-3$. Hence we get that $\phi(\alpha_n)\to \phi(\alpha)$, which proves the continuity of $\phi$.

\medskip 

\textbf{Bijectivity of $\f$.}
Conversely, let us define $\psi:\partial \CX\to \bdr X$.  Fix $\xi\in \partial \CX$. Let $(\^h_n)_{n\in\N}$ be a  sequence converging to $\xi$. In particular, $(\^h_n)_{n\in\N}$ is unbounded so that extracting a subsequence we may and shall assume that $d_{\CX}(\^h_n,\^h_{m})\geq 3$, and therefore $\^h_n$ and $\^h_{m}$ are strongly separated by Remark \ref{rem:3->ss}, for each $n\neq m \in\N$.

We claim that for each $k\in \N$ there exists exactly one element in $\{h_{k}, h_{k}^*\}$ containing infinitely many of the $\{\^h_{n}: n\in \N\}$ and proceed by contradiction:
Suppose that for some $k$ the conclusion fails. Then there would exist $\^h_N$ and $\^h_M$, for $N$ and $M$ arbitrarily large, which are on opposite sides of $\^h_k$. Then using the second part of Lemma \ref{lem:hierarchyGromov}  we would get $(\^h_N|\^h_M)_{\^h_k}\leq 3$, contradicting the assumption that  $(\^h_n)_{n\in\N}$ is a Gromov sequence.

For each $n$, we define $h_n$ as the unique half-space bounded by $\^h_n$ containing infinitely many of the hyperplanes $\{\^h_m\}$. Then we get that $\bigcap_{n\in\N} h_n$ is a non-empty subset of the Roller boundary. 
Furthermore, since $\{\^h_n\}$ are pairwise strongly separated, the same is true for the chosen  half-spaces $h_n$. Hence by \cite[Lemma 5.12]{FLM} their intersection is a singleton and is a regular point, which we define as $\psi( (\^h_n)_{n\in\N})=\psi(\xi)$ (and again, it follows that the boundary $\bdr X$ is non-empty). The point $\psi(\xi)$ does not depend on the choice of the sequence $(\^h_n)$: indeed, suppose $(\^k_n)$ is another sequence converging to $\xi$ and fix $i\geq 0$. Then since $(\^k_n)$ is equivalent to $(\^h_n)$, the same argument as above shows that  there are infinitely many $n\in\N$ such that $\^k_n$ is on the same side of $h_i$ as $h_n$ (if not, we would get $(\^k_n|\^h_m)_{\^h_i}\leq 3$, contradicting that the two sequences are equivalent). This proves that $\psi( (\^k_n)_{n\in\N})$ belongs to $h_i$. Since $i$ is arbitrary, it follows that $\psi( (\^k_n)_{n\in\N})=\psi((\^h_n)_{n\in\N})$, hence $\psi(\xi)$ is well-defined.

 Before proving $\psi$ is continuous, we first show that $\psi$ and $\f$  are formal inverses of each other. We begin by showing that $\psi\circ \varphi=\Id$.  If either one of $\partial \CX$ or $\bdr X$ is empty, then we have seen that the other one also is and the result is clear. If not, 
 let $\alpha\in\bdr X$, and write $\{\alpha\}=\bigcap_{n\in\N}{h_n}$, where $(h_n)_{n\in\N}$ is a decreasing chain of strongly separated half-spaces. Then $\xi=\varphi(\alpha)=\lim \^h_n\in \partial\CX$. Since each $h_n$ contains all the  hyperplanes $\^h_m$ for $m\geq n$, we also get that $\psi(\xi)=\bigcap_{n\in\N} h_n=\{\alpha\}$, i.e. $\psi(\f(\a)) = \a$.

 Now we show that $\varphi\circ\psi=\Id$. Let $\xi\in \partial \CX$; write $\xi=\lim \^h_n$ and choose $h_n$ as above, so that $(h_n)$ is a decreasing chain of strongly separated half-spaces and $\psi(\xi)=\bigcap_{n\in\N} h_n$.  It follows that $\varphi(\psi(\xi))=\lim \^h_n=\xi$.

\medskip

\textbf{Continuity of $\f^{-1}$.}
To prove that $\psi$ is continuous, consider a sequence $(\xi_n)_{n\in\N}\in(\partial \CX)^\N$ converging to $\xi\in\partial\CX$. Let $h$ be a half-space containing $\psi(\xi)$, we have to prove that $\psi(\xi_n)$ is in $h$ for all $n$ sufficiently large. Fix  $A>0$. We know that there is an $N$ such that  $(\xi_n|\xi)_{\^h}>A$ for $n>N$. Now we can write $\xi_n$ as the limit of some hyperplanes bounding half-spaces containing $\psi(\xi_n)$.  In particular there exist  half-spaces $k'\subset k$ containing $\psi(\xi_n)$ and such that $(\^k|\xi_n)_{\^h}>A$ and $(\^k'|\xi_n)_{\^h}>A$. It follows that there exists a constant $C$ (depending only on $\delta$) such that $(\^k|\xi)_{\^h}>A-C$ and $(\^k'|\xi)_{\^h}>A-C$. Taking $A$ large enough, it follows that $k$ (resp. $k'$) and $h$ are not transverse and that $\^k$, $\^k'$ and $\xi$ are on the same side of $\^h$ and therefore contained in $h$. Since $\^k'\subset k$ it follows that $k\subset h$, and as $\psi(\xi_n)\in k$ we indeed get that $\psi(\xi_n)$ converges to $\psi(\xi)$.

 Finally, we note that the equivariance of $\varphi$ and $\psi$ is straightforward. 
 \end{proof}

The question of when the contact graph is unbounded was addressed by Hagen in \cite{HagenSimplicial}, where he deduces unboundedness in a variety of cases with a standing assumption that the cubical complex be locally compact and does not have an infinite family of pairwise crossing hyperplanes (see Sections 4 and 5). By imposing finite dimensionality, we are able to give the following characterization, which has no additional assumptions. We note however that the existence of a regular rank-1 isometry as in item (2) is guaranteed when $\Aut X$ is essential and nonelementary on $X$ \cite{CapraceSageev}.

\begin{cor}\label{Cor: CX unbounded, loxs}
Let $X$ be a finite dimensional irreducible CAT(0) cubical complex. The following hold:
\begin{enumerate}
\item The contact graph $\CX$ is unbounded if and only if there exists an infinite descending chain of strongly separated half-spaces. 
\item An element in $\Aut (X)$ is regular rank-1 on $X$ if and only if it  is loxodromic on $\CX$. 
\end{enumerate}
\end{cor}

\begin{proof}

The hyperbolic graph $\CX$ is unbounded if and only if $\partial\CX\neq \varnothing$. By Theorem \ref{thm:bdrCX}, this is equivalent to $\bdr X\neq \varnothing$, which by definition is equivalent to the existence of  an infinite descending chain of strongly separated half-spaces. This proves (1).

Now let us prove (2). Suppose that $g\in \Aut X$ is regular rank-1 on $X$. Then $g$ is loxodromic on $\CX$ if and only if all of its powers are. Without loss of generality, assume that $g h\subset h$ are strongly separated half-spaces. Then $\Cap{n\in \N}{}g^n h=\{\a_+\}$ and $\Cap{n\in \N}{}g^{-n} h^*=\{\a_-\}$ are the two unique fixed points of $g$ in $\partial_{reg} X$ and hence  by Theorem \ref{thm:bdrCX} it follows that $\phi(\a_\pm)$ are the two unique fixed points for $g$ in $\partial \CX$. Next, applying Lemma \ref{lem:hierarchyGromov} to $\{g^n\^h\}_{n\in \N}$ we deduce that $d_\CX (g^n\^h, \^h) = (g^n\^h, \^h)_{\^h} \geq n-3$. Dividing by $n$ and taking the limit, we deduce that $\ell(g)>0$ and $g$ is loxodromic.

Conversely, suppose that $g\in \Aut X$ is not regular rank-1 on $X$. Then, one of the two following options hold and we conclude the proof in each case:
\begin{description}
\item[(a)] the element $g$ is  elliptic on $X$ or,
\item[(b)] it is loxodromic on $X$ but not regular rank-1. 
\end{description}

\noindent
{\bf(a):} If $ g$ is elliptic then by  \cite{Haglund}, there is a $\<g\>$-fixed point $x$. Projecting to $\CX$ we get that  $\<g\>$-orbit of  $\pi(x)$ is bounded and hence $g$ is elliptic in $\CX$.

\noindent
{\bf(b):} Suppose $g$ is loxodromic on $X$ and not regular rank-1. Then the same is true for all powers of $g$. Furthermore, since $g$ is loxodromic, then there exists an $h\in \frakH$ and $p\in \Z$ with $g^ph\subset h$ and these are not strongly separated. Therefore, for every $n\in \N$ we have that $g^{np}h\subset h$  are not strongly separated, meaning that $d_\CX(g^{np}\^h, \^h)\leq 2$ and hence the $\<g\>$ orbit of $\^h$ is bounded in $\CX$ . 
\end{proof}

\begin{remark}
It is possible to use Theorem \ref{thm:bdrCX}, together with Theorem \ref{thm:MaherTiozzo} below, to give a simpler proof of one of the main results of \cite{FLM}, namely, the convergence of the random walk on $X$ to a regular point in the Roller boundary.
\end{remark}

We conclude this section by comparing the various notions of elementary actions. Recall that an action on a CAT(0) cubical complex is called nonelementary if it has no finite orbit on $X\cup \partial X$. For groups acting on hyperbolic spaces, we use the following definition.

\begin{definition}
An action of a group on a hyperbolic space is called \emph{nonelementary} if there are two elements $g$ and $g'$ of the group which act as loxodromic isometries with disjoint fixed points on the Gromov boundary.
\end{definition}

Equivalently, the action of $G$ on a hyperbolic space $\mathcal Z$ is nonelementary if and only if $G$ does not have a bounded orbit in $\mathcal Z$ or an orbit of cardinality less than $2$ in $\partial \mathcal Z$ (see for example \cite[Theorem 6.2.3 and Proposition 6.2.14]{Nonproper}).

The following addresses the question of comparing when actions on $X$ versus actions on $\CX$ are nonelementary. It also is a reformulation of Caprace and Sageev's Rank-Rigidity Theorem in terms of the action on the contact graph. While item (1) of the theorem does not assume the action is essential, nor nonelementary on $X$, in practice, those would be a reasonable way to conclude that the $G$ action on $\CX$ is nonelementary, which is essentially Caprace and Sageev's Rank-Rigidity Theorem. It is worth noting that Hagen also has an adaptation of Caprace and Sageev's Rank-Rigidity Theorem for the contact graph, see \cite[Theorem 5.4]{Hagen}.

\begin{proposition}[Caprace-Sageev Rank-Rigidity on $\CX$]\label{prop:nonelementCX}
Let $X$ be a finite dimensional CAT(0) cubical complex. Consider an action $G\to\Aut(X)$ with $X$ finite dimensional. The following are true:

\begin{enumerate}
\item If the action of $G$ on $\CX$ is nonelementary then $X$ is irreducible and $G$ contains two regular rank-1 elements with disjoint fixed-point sets in $\partial_{reg} X$. 
\item If the $G$ action on $X$ is essential and nonelementary, then either $X$ is a product or the action of $G$ on $\CX$ is nonelementary. 
\end{enumerate}
\end{proposition}

\begin{proof}

Assume the action of $G$  on $\CX$ is nonelementary.  Then, there exist two elements $g,g'\in G$ which act loxodromically on $\CX$ and have disjoint fixed point sets. By Corollary \ref{Cor: CX unbounded, loxs}, it follows that $g$ and $g'$ are regular rank-1 on $X$ and by Theorem \ref{thm:bdrCX} their fixed point sets are disjoint. Furthermore, this also means that $\CX$ is unbounded and hence $X$ is not a product, i.e. it is irreducible, which proves (1).

To prove (2) assume the action of $G$ is essential, and nonelementary and that $X$ is irreducible. By \cite{CapraceSageev}, there exists a quadruple of pairwise facing super strongly separated half-spaces $h_1, h_2, h_1',h_2'$. Applying their Double Skewering Lemma, we can find $\g, \g'\in G$ such that $\g h_2\subset h_1^*\subset h_2$, and $\g' h_2'\subset h_1'^*\subset h_2'$. By \cite[Proposition 4.3]{BeyFio}, $\g$ has only two fixed points in $\bdr X$, which are contained in $h_1^*$ and $h_2^*$. Similarly $\g'$ has two fixed points, which are contained in $h_1'^*$ and $h_2'^*$. Using Theorem \ref{thm:bdrCX}, we conclude that indeed $\g$ and $\g'$ are loxodromic on $\CX$ with disjoint fixed points.
\end{proof}

\section{Random walks}

From now on, we fix a group $G$ acting on $X$ by automorphisms and assume that the action of $G$ is nonelementary. By \cite{CapraceSageev}, $X$ contains a copy of its $G$-essential core as a $G$-invariant sub CAT(0) cubical complex, possibly after cubical subdivision. Since we are only interested in the behavior of the orbit of some vertex, we shall replace $X$ by its $G$-essential core and assume the action is essential. Note that by \cite[Proposition 2.26]{CFI} any action $G\to \Aut X$ is either Roller elementary (i.e. has a finite orbit in the Roller compactification) or there is a finite index subgroup $G_0\leq G$ and a $G_0$-equivariant quotient of $X$ that is nonelementary and essential. However, we shall stick to these standard assumptions.

Let $\mu\in \Prob(G)$ be such that the support of $\mu$ generates $G$ as a semigroup. The \emph{random walk} associated to $\mu$ is the following process: we draw some elements $g_i$, independently, of law $\mu$, and form the product $Z_n=g_1\cdots g_n$. 

More formally, let $\Omega=\G^{\N}= \G\times \G^{\N^*}$ and $\Ps$ be the probability measure on $\Omega$ defined by $\Ps=\delta_e\times \mu^{\N^*}$. The space $\Omega$ is the \emph{space of increments}. If $\omega\in \Omega$, we denote by $g_i(\omega)$ the $i$th element of the sequence $\omega$. We often omit the $\omega$ and write only  $g_i$. Then we consider 
$$Z_n:=Z_n(\omega)= g_1(\omega)\cdots g_n(\omega)=: g_1\cdots g_n.$$

We are interested in the behavior of $Z_n\cdot o$, where $o$ is our choice of base-vertex in $X$. This behavior will depend in particular on integrability properties of $\mu$. Recall that the measure $\mu$ has  finite $n$th moment (relative to the metric $d$) if $\int_G d(go,o)^n d\mu(g)<+\infty$.

Finally, recall that a measure $\nu\in \Prob(\partial X)$ is \emph{stationary} if $\nu = \sum \mu(g) g_*\nu$. The following was proved in  \cite[Corollary 7.3]{FLM}.

\begin{theorem}\label{thm:uniquestat}
Assume that the action of $G$ on $X$ is nonelementary and essential. Then there exists a unique stationary measure on $\overline X$, which is supported on $\partial X$.
\end{theorem}

In the following, we will denote this stationary measure by the letter $\nu$. The stationary measure for the reflected measure $\check \mu$ (defined by $\check\mu(g)=\mu(g^{-1})$) will be denoted $\check \nu$.

\subsection{The Random walk on the cubical complex}

Assume that $\mu$ has finite first moment. Then one can define the \emph{drift} of the random walk as follows.

\begin{definition}
The \emph{drift} of the random walk on $X$ is 
$$
\lambda= \underset{n}{\inf} \frac{1}{n}\int_\Omega d(Z_n(\omega)o,o)\,d\Ps(\omega)= \underset{n}{\inf}\frac {1}{ n} \int_G d(o,go) d\mu^{*n}.
$$
Equivalently, by  a standard application of Kingman's Subbaditive Ergodic Theorem, one can define $\lambda$ as the (almost surely defined) limit
$$\lambda = \Lim{n\to \8}\frac{1}{n} d(Z_n o,o).$$
\end{definition}

\begin{remark}\label{Rem: ch drift is drift}
Let $\check \mu$ be the measure defined by $\check \mu(g)=\mu(g^{-1})$, and $\check Z_n$ be the associated \emph{right} random walk with drift $\ch \lambda$. Then $\lambda=\ch \lambda$:

\begin{eqnarray*}
\lambda &=& \underset{n}{\inf}\frac {1}{ n} \int_G d(o,go) d\mu^{*n}\\
&=& \underset{n}{\inf}\frac {1}{ n} \int_G d(g^{-1}o,o) d\mu^{*n}\\
&=& \underset{n}{\inf}\frac {1}{ n} \int_G d(go,o) d\ch\mu^{*n}=\ch \lambda.
\end{eqnarray*}

\end{remark}

From now on, $\lambda$ (respectively $\ch \lambda$) will denote the drift of the $\mu$ (respectively $\ch \mu$) random walk on $X$. By \cite[Theorem 1.2]{FLM}, if the action of $G$ on $X$ is nonelementary and essential, we always have $\lambda,\ch\lambda >0$.

\begin{proposition}\label{prop:cplcvg}
Assume that $X$ is finite dimensional and the action of $G$ on $X$ is nonelementary and essential, and that $\mu$ has  finite second moment. Then for every $\eps>0$ there exist $C_n$ such that $\sum_{n\geq 1} C_n <+\infty$ and for every $x\in \overline X$
\begin{eqnarray*}
\Ps (| h_x(Z_n^{-1} o)-n\lambda | \geq \eps n )&\leq C_n\\
\Ps (| h_x(Z_n o)-n\lambda | \geq \eps n )&\leq C_n\\
\Ps( | d(Z_n o,o)-n\lambda| \geq \eps n)&\leq C_n.
\end{eqnarray*}

\end{proposition}

\begin{proof}

 Let $\nu$ (resp. $\check\nu$) be the $\mu$-stationary (resp. $\check\mu$-stationary) measure on $\partial X$.

Recall that the map $\sigma:(g,\xi)\mapsto h_\xi(g^{-1} o)$ is a cocycle (see Remark \ref{rem: horo is cocycle}), and $\int_{G\times \overline X} \sigma(g,\xi) d\mu(g)d\nu(\xi)=\int_{G\times \overline X} \sigma(g,\xi) d\check \mu(g)d\check \nu(\xi)=\lambda$ (using Remark  \ref{Rem: ch drift is drift} above). Note also that$|h_x(Z_no)| \leq d(o,Z_no)$ by the triangle inequality, so that $\int_{G\times \overline X} \sup_{x\in \overline X} |\sigma(x,\xi)|^2d\mu(g)<+\infty$ (and similarly for $\check\mu$). 
 Therefore the first and second inequality follow from \cite[Proposition 3.2]{BQLinear}, applied to the measures $\check \mu$ and $\mu$ respectively.
The third inequality then follows from the second inequality and \cite[Proposition 9.4]{FLM}.
\end{proof}

\subsection{Random walk on the contact graph}

Random walks on hyperbolic spaces have been intensively studied. One can prove in particular the convergence to the boundary of the random walk. For non-proper spaces, this was first obtained by  Maher and Tiozzo \cite{MaherTiozzo}. Recently Gou\"ezel \cite{Gouezel} improved on their result and got rid of the moment assumptions.

\begin{theorem}\label{thm:MaherTiozzo}
Assume that $G$ has a nonelementary action on a hyperbolic space $(Y,d_Y)$. Let $o\in Y$. Let $\mu$ be an admissible measure on $G$ with finite first moment (relative to $d_Y)$, and consider the random walk $(Z_n)$ associated to $\mu$. 

Then there exists $\lambda_0>0$ (possibly infinite) such that $d_Y(o,Z_n o)/n\to \lambda_0$ almost surely.

Furthermore, for every $r<\lambda_0$ there exists $\kappa>0$ such that $$\Ps(d_Y(o,Z_n o)<r n)<e^{-\kappa n}.$$
\end{theorem}

Now let us go back to our situation: 
the group $G$ acts non-elementarily on the cubical complex $X$, and we consider the random walk $Z_n.o$ with $o\in X$. 

Let $S(n)$ be the maximal number of pairwise strongly separated hyperplanes separating $o$ from $Z_n o$.

\begin{corollary}\label{SSSgrowth}
Assume that $X$ is finite dimensional and irreducible and that $G$ has a nonelementary essential action on $X$.
Then there exists $0<\lambda_0\leq +\infty$ such that $\liminf S(n)/n\geq \lambda_0/3$ almost surely.

Furthermore there exists $A>0$ and  $C_n>0$, with $\sum C_n<+\infty$, such that 
$$\Ps(S(n)<A n)<C_n$$

\end{corollary}

\begin{proof}
By Proposition \ref{prop:nonelementCX} the action of $G$ on $\CX$ is nonelementary. Hence we can apply Theorem \ref{thm:MaherTiozzo} and deduce that $d_{\CX}(Z_n\pi(o),\pi(o))/n$ tends to $\lambda_0>0$, and that there exists $A>0$ such that the probability $\Ps(d_{\CX}(\pi(Z_no),\pi(o))<An)$ decreases exponentially in $n$. The result then follows from Lemma \ref{lem:projections}.
\end{proof}

\section{Proof of the Central Limit Theorem}

Our goal is now to prove the promised Central Limit Theorem for cubical complexes.
 As mentioned in the introduction, we shall make explicit use of the results and strategies from the work of Yves Benoist and Jean-Fran\c cois Quint on the Central Limit Theorem \cite{BenoistQuintHyp,BQLinear}. Specifically, they prove a Central Limit Theorem for groups that act ``non-elementarily" on $\R^d$ by linear automorphisms, or on proper Gromov hyperbolic metric space by isometries. 
 
 The general idea of their proof (and therefore of ours) is to first relate the distance $d(o,Z_n o)$ to a cocycle $\sigma : Z\times G\to \R$. In each case the space $Z$ is a relevant boundary and the cocycle can be interpreted as a suitable horofunction; the main point being that this horofunction is at a bounded distance from the quantity $d(o,Z_n o)$. Then the main idea is to use the cocycle relation to produce a martingale for which the Central Limit Theorem is well-known. The relevant criterion is quoted in Theorem \ref{thm:critereTCL} below.
 
 However, in order to do so, one must first ``correct" the cocycle by the coboundary of a bounded  function in order to indeed obtain a martingale. In the previous proofs (see eg \cite{LePage_82} or \cite{Led01}) this was obtained by a spectral argument, but the proof of Benoist and Quint gives an explicit integral formula for this correction. The difficulty is then to check that this formula is well-defined (i.e. that the integral converges). In the case of hyperbolic groups, this is done in \cite[Proposition 4.1 and 4.2]{BenoistQuintHyp}. Here while we follow the same general argument, the geometric content of these estimates is more involved. This is the content of Section \ref{sec:geomestimates}, which culminates with the proof of Proposition \ref{prop:integrability}.

%

\subsection{Some geometric estimates}\label{sec:geomestimates}

 This strategy relies on some estimates of the speed of the random walk in various directions. In order to unfold these estimates, we first isolate a lemma in cubical geometry.

\begin{lemma}\label{box} 
 Let $X$ be finite dimensional and fix  $o\in X$ and $x,y,z\in \overline X$. Let $m_1= m(o,z,y), m_2= m(o,z,x), m_3 = m(o,m_1, m_2), m_4= m(o, x,y)$. If there is a pair of strongly separated half-spaces $h_1\subset h_2$ such that $z, m_2 \in h_1$ and $o, m_3\in h_2^*$ then $m_1=m_3=m_4$.
\end{lemma}

 The reader may visualize a general configuration of the points as follows:
\begin{center}
 \includegraphics[]{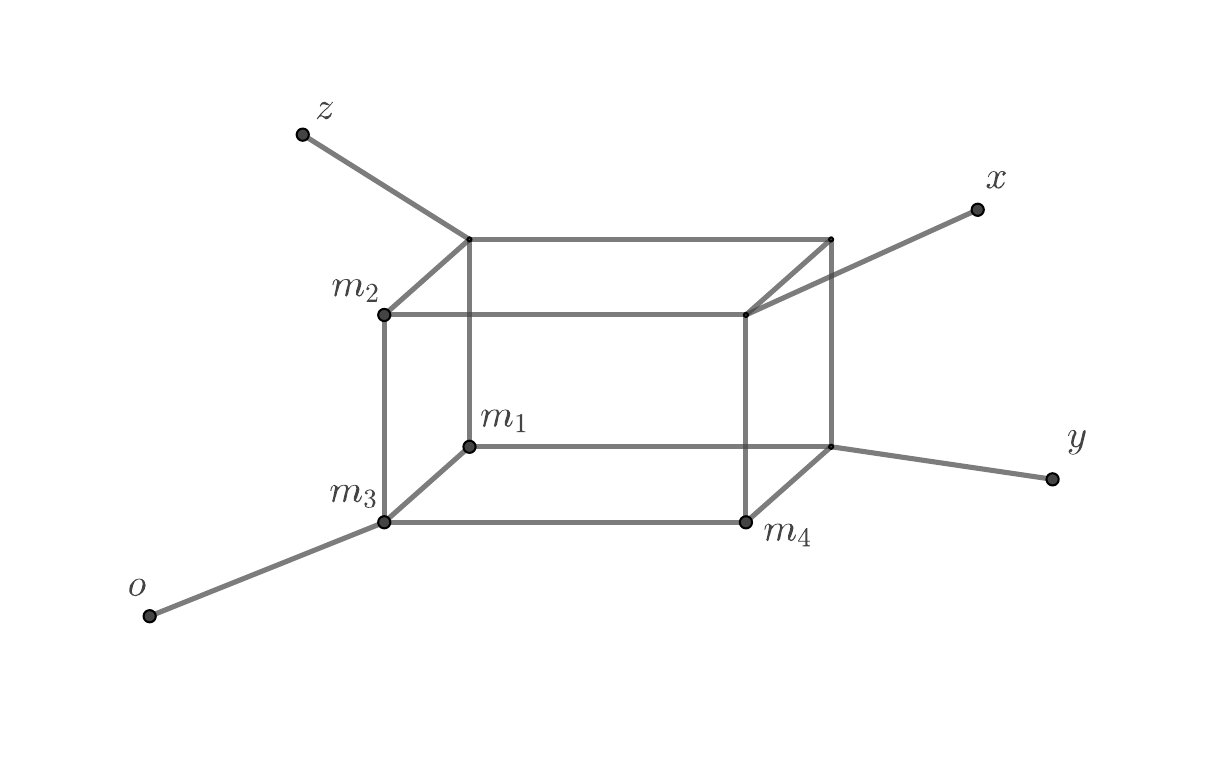}
 \end{center}

 \begin{proof}
  We begin with an observation, which directly follows from the definition of the median: Let $a,b, c\in \overline X$, and $m=m(a,b,c)$, $h\in \frakH$ such that $a\in h$. It follows that $m\in h^*$ if and only if $b,c\in h^*$. We will apply this observation multiple times below. 
  
  To simplify the exposition, we  shall say that two points $a,b$  are \emph{on the same side} if either $a,b \in h_1\subset h_2$ or $a,b \in h_2^*\subset h_1^*$ or  \emph{on the other side} if either $a \in h_1\subset h_2$ and $b \in h_2^*\subset h_1^*$, or vice versa. 
  
Our hypotheses state that $z, m_2 \in h_1\subset h_2$ and $o,m_3\in h^*_2\subset h^*_1$ and  so
\begin{itemize}
\item  $m_2=m(o,z,x)$ is on same side as $z$ and both are on the other side of $o$ and hence $x\in h_1\subset h_2$. 
\item  $m_3 = m(o,m_1, m_2)$ is on the same side as $o$ and both are on the other side of $m_2$ and hence $m_1\in h^*_2\subset h^*_1$.
\item $m_1 = m(o,z,y)$ is on the same side as $o$ and on the opposite side of $x$ and hence $y \in h^*_2\subset h^*_1$.
\item $o$ is on the same side as $y$ and hence $m_4 \in h^*_2\subset h^*_1$ (we include for completeness, but we will not need it). 
\end{itemize}

To summarize we have:

$$x,z,m_2\in h_1\subset h_2\text{ and } o,y,m_1, m_3, m_4\in h^*_2\subset h^*_1$$
  
We proceed by contrapositive. We show that if $m_1\neq m_3$ or $m_1\neq m_4$ then $h_1$ and $h_2$ are not strongly separated. 

If $m_1 \neq m_3$ then there is a half-space $k\in \frakH$ such that $m_1\in k$ and $m_3\in k^*$. Arguing as above we have that 

\begin{itemize}
 \item Since $m_3 = m(m_1, m_2, o)$ we deduce that $m_2, o\in k^*$
 \item Since $m_1 = m(o,z,y)$ we deduce that $y,z\in k$. 
 \item Since $m_2 = m(o,z,x)$ we deduce that $x\in k^*$.
\end{itemize}

Therefore, for each $i= 1, 2$ we have that $m_1 \in k\cap h_i^*\neq \varnothing$, $o\in k^*\cap h_i^*\neq \varnothing$, $z\in k\cap h_i\neq \varnothing$, and $x\in k^*\cap h_i\neq \varnothing$, i.e. $h_i\pitchfork k$ and so $h_1$ and $h_2$ are not strongly separated.

Before continuing, we note that by the median axioms (see Remark \ref{Rem: MedianAxioms}) we have: 
$$m_3= m(o,m(o,z,y),m(o,z,x)) = m(o,z,m(x,y,o)) = m(o,z,m_4).$$

Next, suppose that $m_3\neq m_4$. Then there exists $\ell \in \frakH$ such that $m_3 \in \ell$ and $m_4\in \ell^*$. Arguing as above we conclude $o,z\in \ell$, and since $m_4 = m(o,z,y)$ it follows that $y\in \ell$. And, since $m_1= m(o,z,y)$ we conclude that $m_1\in \ell^*$. 

Hence, it follows that $h_i\pitchfork \ell$ for $i =1,2$ using the four points $o,m_1,y,z$ and the above calculations to show that the four appropriate intersections are not empty. Therefore, $h_1$ and $h_2$ are not strongly separated. 

\end{proof}

The following is the quantitative heart of the paper. It will be applied later to the sequence $g_n=Z_n$. Its assumptions are the ones that we know to hold with high probability (by Proposition \ref{prop:cplcvg} and Theorem \ref{thm:MaherTiozzo}): the random walks goes to infinity at linear speed, both in $\CX$ and in $X$; and the cocycle $h_x(Z_n^{-1} o)$ also 
goes to infinity at linear speed. This will allow us to give a quantitative lower bound on the Gromov product $(Z_n x|y)_o$, which we will use in Proposition \ref{prop:integrability} to finally deduce the integrability needed in Lemma \ref{lem:centerable} below.

\begin{lemma}\label{lem:boite} Assume that $X$ is finite dimensional.  Let $(g_n) \in G^\N$, and $S(n)$ be the maximal number of pairwise strongly separated hyperplanes separating $o$ from $g_n o$. Fix $x,y\in \overline X$.
Assume that there exist $\lambda, \eps, A>0$  such that
\begin{itemize}
\item $S(n)\geq A n$
\item
 $|h_x(g_n^{-1} o) - n\lambda |\leq \eps n;$
\item $|d(g_n o, o) - n\lambda |\leq \eps n;$
 \item $|h_y(g_n o) - n\lambda |\leq \eps n.$ 
\end{itemize}
Then  we have
\begin{enumerate}
\item $(g_n o| y)_o \leq \eps n$
\item  $(o|x)_{g_n^{-1} o} \geq (\lambda-\eps) n$.
\end{enumerate}

If in addition, $A>3\eps$ then for $n$ sufficiently large we have that
\begin{enumerate}
\item [(3)]$(g_n x|y)_o \leq \eps  n$.
\end{enumerate}
\end{lemma}

\begin{proof}[Proof of (1) and (2)]
Recall that horofunctions are normalized so that $h_x(o)=0$ for every $x\in \overline X$.
Applying the last equality of Remark \ref{rem:Gproduct}, we have that 
 
 \begin{eqnarray*}
0\leq 2(g_n o | y)_o &= |d(o,g_n o)-h_y(g_n o)  -\lambda n + \lambda n|\\
&\leq | d(o,g_n o)-\lambda n| + |h_y(g_n o)   - \lambda n | \\
&\leq 2\eps n.
\end{eqnarray*}

For the second inequality we apply the last equation of Remark \ref{rem:Gproduct}:

$$2(o|x)_{g_n^{-1} o} = d(g_no,o)+h_x(g_n^{-1} o)-h_x(o) \geq -2\eps n+2\lambda n.$$
\end{proof}
\begin{proof}[Proof of (3)]

Let $m_1=m(o,g_n o,y)$, $m_2=m(o,g_n o,g_n x)$,  $m_3=m(m_1,m_2,o)$, and $m_4=m(o,g_nx,y)$.

  By part (1) we have $d(o,m_1)=(y|g_n o)_o\leq\eps n$. Similarly by part (2) we have $d(o,m_2)=(g_n o| g_n x)_o=(o| x)_{g_n^{-1} o}\geq (\lambda -\eps) n$, and we therefore get
 \begin{eqnarray*}
d(m_2,m_3)&=d(o,m_2)-d(o,m_3)\\
&\geq d(o,m_2)-d(o,m_1)\\
&\geq (\lambda -2\eps )n.
\end{eqnarray*}

 Note that $m_2, m_3\in X$, $m_3\in I(o,m_2)$, and $m_2\in I(o,g_no)$. Therefore, there exists a geodesic from $o$ to $g_n o$ which passes first through $m_3$ and then through $m_2$, i.e.
 $d(o,g_no)= d(o,m_3)+d(m_3,m_2)+d(m_2,g_n o)$. Hence
 
 \begin{align*}
  d(o,m_3) +d(g_no,m_2) &= d(o,g_n o) - d(m_3,m_2)\\
 &\leq (\lambda + \eps) n - (\lambda-2\eps) n\\
 &\leq 3 \eps n.
 \end{align*}

   Since $S(n)>An$, there exists a family $(h_k)_{k=1}^{\lfloor An\rfloor}$ of $\lfloor An\rfloor$-many pairwise strongly separated hyperplanes which are between $o$ and $g_n o$. Assume that $A>3\eps $. Then for $n$ large enough we have $\lfloor An\rfloor >3\eps n+2$. Among the $d(o,g_n o)$-many hyperplanes separating $o$ from $g_n o$, there are at most $3\eps n$ of them separating either $o$ and $m_3$ or $g_n o$ and $m_2$, and the rest of them separate $m_2$ from $m_3$. Hence  there exist two hyperplanes from the family $(h_k)$ which separate $m_2$ from $m_3$. 
   In particular, there exist at least two strongly separated hyperplanes separating $m_2$ from $m_3$.
   By Lemma \ref{box} we conclude that $m_1=m_3=m_4$.
  Hence $(g_n x|y)_o =d(o,m_4) = d(o,m_1)\leq \eps n$.

\end{proof}

\begin{proposition}\label{prop:integrability}
Let $X$ be irreducible and $G\to \Aut(X) $ a nonelementary and essential action.
Assume also that $\mu$ has finite second moment. Then there exist $a>0$ and $C_n$ such that $\sum_{n\geq 1} C_n <+\infty$ and for every $x,y\in \overline X$  we have, for every $n$:
\begin{eqnarray*}
\Ps ( (Z_no|Z_n x)_o\leq an) &\leq C_n\\
\Ps( (Z_no|y)_o\geq an)&\leq C_n\\ 
\Ps( (Z_n x| y)_o\geq an) &\leq C_n
\end{eqnarray*}
\end{proposition}

\begin{proof}
Choose first $A$ as in Corollary \ref{SSSgrowth}, and the sequence $(C'_n)$ accordingly. Now choose $\eps<\min(A/3,\lambda/2)$, then apply Proposition~\ref{prop:cplcvg} and obtain again some summable sequence $(C'_n)$ which we can assume to be the same as the previous one. It follows that the union  of the four events defined by the inequalities
$S(n)< A n$;
 $|h_x(Z_n^{-1} o) - n\lambda |\geq \eps n;$
$|d(Z_n o, o) - n\lambda |\geq \eps n;$ and  $|h_y(Z_n o) - n\lambda |\geq \eps n$ 
has a probability which is bounded above by $4C'_n$. Letting $C_n=4C'_n$ we note that this is still a summable sequence. By Lemma \ref{lem:boite}, if $Z_n$ is not in this union then we have  $(Z_no|Z_n x)_o\geq (\lambda-\eps) n$, $(Z_no|y)_o\leq \eps n$ and $(Z_n x| y)_o\leq \eps n$. 

Therefore choosing $a\in (\eps, \lambda-\eps)$ guarantees that the union of the three  events $(Z_no|Z_n x)_o\leq an$, $(Z_no|y)_o\geq an$, $(Z_n x| y)_o\geq an$ have probability bounded by $C_n$, which completes the proof. 

\end{proof}

\subsection{Central Limit Theorem}\label{Sect: CLT}

We can now turn to the proof of our main theorem. 

 We will need the following few definitions.
Let $G$  act continously on a compact metrizable space $Z$, and $E$ be a real vector space.

\begin{definition}
Let $\sigma:G\times Z\to E$ a continous cocycle. 
\begin{itemize}
\item We say that $\sigma$ has \emph{constant drift} if $\int_G \sigma(g,x) d\mu(g)$ does not depend on $x\in Z$. The \emph{average} of $\sigma$ is then  $\int_G \sigma(g,x) d\mu(g)$
\item We say that $\sigma$ is \emph{centerable} if there exists a bounded Borel  map $\psi : Z\to E$   and a cocycle $\sigma_0:G\times Z\to E$  with constant drift such that for every $(g,x)\in G\times Z$ 
\begin{enumerate}
\item 
$\sigma(g,x)=\sigma_0(g,x)+\psi(x)-\psi(gx)$
\item $\sigma_0$ has constant drift.
\end{enumerate}
The \emph{average} of $\sigma$ is then defined as the average $\sigma_\mu = \int_G \sigma_0(g,x) d\mu(g)$ of $\sigma_0$.
\end{itemize}
\end{definition}

If $\nu$ is a stationary measure on $Z$, then the average of a centerable cocycle does not depend on the choice of $\sigma_0$ and $\psi$, as we have 
$\int _{G\times Z} \sigma(g,x) d\mu(g)d\nu(x) = \sigma_\mu +\int_Z \psi(x) d\nu(x)- \int_{G\times Z} \psi(gx)d\mu(g)d\nu(x) = \sigma_\mu $ by stationarity of $\nu$.

Even though the Central Limit Theorem we are aiming for is one-dimensional, in order to treat the case when $X$ is reducible, we are led to study  a multidimensional version of it. Let us first introduce the relevant notations.
Let $E$ be a finite-dimensional real vector space. We denote $S^2E$ the space of \emph{symmetric 2-tensors} on $E$: it is the subspace of $E\otimes E$ which is invariant by the flip $x\otimes y\mapsto y\otimes x$. This space can be identified to  quadratic forms on $E^*$, by the following formula: $(x\otimes y+y\otimes x) (f)=2f(x)f(y)$. In particular, for $v\in E$, we write 
 $v^2\in S^2E$ for $v\otimes v$, that is, the element defined by $v^2(f)=f(v)^2$ for $f\in E^*$.
  If $\Phi\in S^2E$ is non-negative then one can define the multidimensional centered Gaussian law $\mathcal N(0,\Phi)$ of covariance 2-tensor $\Phi$. One possible definition of this law is as follows: if $X$ is a random vector then $X$ follows the law $\mathcal N(0,\Phi)$ if and only if for every $f\in E^*$, $f(X)$ follows a (one-dimensional) centered Gaussian law of variance $\Phi(f)$. 

The following theorem is an application of the Central Limit Theorem for martingales \cite{Brown}, \cite[Theorem 3.4]{BQLinear}. 

\begin{theorem}\label{thm:critereTCL}
Let $G$ be a contable group acting by homeomorphisms on a compact metrizable  space $Z$. Let $E$ be a finite-dimensional real vector space and $\sigma:G\times Z\to E$ be a continous cocycle. Assume that $\sigma$ is centerable with average $\sigma_\mu$, that $\int_G sup_{x\in Z} \Vert \sigma(g,x)\Vert^2 d\mu(g)<+\infty$, and that there exists a unique $\mu$-stationary probability measure $\nu$ on $Z$.  

Then the random variable $\frac{\sigma(Z_n,x)-n\sigma_\mu}{\sqrt n}$ converges in law to a Gaussian law.

Furthermore if we write $\sigma(g,z)=\sigma_0(g,z)+\psi(z)-\psi(gz)$ with $\psi$ bounded and $\sigma_0$ with constant drift, then the  covariance 2-tensor of the limit law is
$$\int_{G\times Z} (\sigma_0(g,x)-\sigma_\mu)^2 d\mu(g) d\nu(x)$$
\end{theorem}

\begin{remark}
In \cite{BQLinear} the map $\psi$ is assumed to be continuous. However, in the proof of \cite[Theorem 3.4]{BQLinear} only the boundedness of $\psi$ is used.
\end{remark}

Now let us go back to our specific situation. We continue to assume that $X$ is a finite-dimensional CAT(0) cubical complex and that $G\to \Aut(X)$ is an essential and nonelementary action. 

\begin{lemma}\label{lem:centerable}
Assume that $X$ is irreducible, and that the action of $G$ on $X$ is nonelementary and essential. Assume also that $\mu$ has  finite second moment. Then the cocycle $\sigma:G\times \overline X \to \R$ defined by $\sigma(g,x)=h_{x}(g^{-1}o)$ is centerable. Its average is the drift $\lambda$.

More precisely we have $\sigma(g,x) = \sigma_0(g,x)+\psi(x)-\psi(gx)$ with $\psi(x) =-2\int_{\partial X} (x|y)_o d \check{\nu}(y)$ and with $\sigma_0$ of constant drift.
\end{lemma}

\begin{proof}
Let $\check \nu$ be the $\check\mu$-stationary measure on $\overline X$. 
Using Proposition \ref{prop:integrability} we deduce as in \cite[Proposition 4.2]{BenoistQuintHyp} (replacing \cite[Lemma 4.5]{BenoistQuintHyp} by Proposition \ref{prop:integrability}) that $\psi(x)=-2\int_{\overline X} (x|y)_o d \check{\nu}(y)$ is finite for every $x\in \overline X$, and furthermore that $\sup_x \psi(x) <+\infty$, that is, $\psi$ is bounded. We also note that $\psi$ is Borel by Fubini's Theorem. 

Now for all $g\in G$ and $x,y\in \overline X$, by Remark \ref{rem:Gproduct}
we have
$$h_x(g^{-1}o) = -2(x|g^{-1}y)_o+2(gx|y)_o+h_y(go)$$ 

Now, by \cite[Proposition 9.4]{FLM}, we have $\lim_n \frac 1n h_\xi(Z_n o) = \lambda$ almost surely, for every $\xi \in \partial X$. Using Birkhoff's Ergodic Theorem together with the cocycle identity, as in the proof of \cite[Theorem 9.3]{FLM}, we also get that 
$\lim_n \frac 1n h_\xi(Z_n o)  = \int_{G\times \partial X} \sigma(g^{-1},x) d\mu(g) d\check \nu(x) $, hence that this integral is equal to $\lambda$. 
Integrating the previous equality on $G\times \partial X$ for the measure $d\mu(g)d\check \nu(y)$ we get 
$\int_G \sigma(g,x) d\mu(g) = \psi(x)-\int_G \psi(gx) d\mu(g)+\lambda$

This means that  the cocycle $\sigma_0:G\times\overline X \to \R$ defined by the formula $\sigma_0(g,x)=\sigma(g,x)-\psi(x)+\psi(gx)$ has the property that if $x\in \overline X$  then
$$\int_G \sigma_0(g,x)d\mu(g)= \lambda.$$
\end{proof}

\begin{theorem}\label{thm:gaussian}
Let $G$ have a nonelementary and essential action on the finite dimensional CAT(0) cubical complex $X$.  Assume that the action of $G$ on $X$ stabilizes each irreducible factor of $X$.
Assume also that $\mu$ has finite second moment.
There exists a Gaussian law $N_\mu$ on $\R$ such that 
$$\frac{d(Z_n o,o)-n\lambda}{\sqrt n}\to N_\mu$$ 
in distribution. 
Furthermore, the variance of this Gaussian law is given by 

$$\int_{G\times \partial X} (h_\xi(g^{-1}o)-\psi(x)+\psi(gx)-\lambda)^2 d\mu(g) d\nu(x)$$

where $\psi(x) =-2\int_{\partial X} (x|y)_o d \check{\nu}(y)$.
\end{theorem}

\begin{proof}
Let $X= X_1\times \cdots \times X_D$ be the irreducible decomposition of $X$ and note that $\~X = \~X_1\times \cdots \times \~X_D$.
For each $1\leq i\leq D$, the  action of $G$ on $X_i$ is again nonelementary and essential, so it has  positive drift $\lambda_i$ \cite[Theorem 1.2]{FLM}. The total drift for the action of $G$ on $X$ is then $\lambda=\lambda_1+\dots+\lambda_D$. Now for each $i$ there is a cocycle $\sigma_i:G\times \~X_i\to \R$ defined by $\sigma_i(g,x)=h_{x_i}(g^{-1}o)$ (where $x_i$ is the $i$th-coordinate of $x$ and so the horofunction is understood to be computed with only the $i$th-coordinate). Define $\~ \sigma:G\times \~X\to \R^D$ by $\~ \sigma=(\sigma_1,\dots,\sigma_D)$. By Lemma~\ref{lem:centerable} we get that $\sigma_i(g,x)$ is centerable for every $i$. Therefore $\~ \sigma$ is centerable, with average $\bar\lambda=(\lambda_1,\dots,\lambda_D)$. Write $\~ \sigma(g,x)=\~\sigma_0(g,x)+\~\psi(x)-\~\psi(gx)$, coordinate-wise with $\~\psi$ bounded and measurable  and  $\~\sigma_0$ with constant drift. Recall that   $\nu$ is the unique stationary measure on $\overline X$ (by Theorem \ref{thm:uniquestat}). It follows from Theorem \ref{thm:critereTCL} and Lemma \ref{lem:centerable}, using also that $\mu$ has a finite second moment (and the fact that $|h_x(g^{-1}o)|\leq d(o,g^{-1}o)$ for every $x\in \overline X$)
that $\frac{\tilde \sigma(Z_n,x)-n\bar \lambda}{\sqrt n}$ converges in law to a Gaussian law whose covariance 2-tensor is
$$\int_{G\times \partial X} (\~\sigma_0(g,\xi)-\bar\lambda)^2 d\mu(g) d\nu(\xi)$$

Let $f:\R^d\to \R$ be the linear form $(u_1,\dots,u_D)\mapsto u_1+\dots+u_D$. Write $ \sigma=f\circ \~\sigma$, $ \sigma_0=f\circ \~\sigma_0$, and $ \psi=f\circ \~\psi$. Note that $ \sigma(g,\xi)=h_\xi(g^{-1}o)$ for $\xi\in \partial X$ and $g\in G$ (so that the notation is consistent with the previous one).

Since the linear projection of a multidimensional Gaussian is again a Gaussian, we get that for $g\in G$ and $x=(x_1,\dots,x_D)\in \overline X$, the ratio $\frac{ \sigma(Z_n,x)-n\lambda}{\sqrt n}$ converges to a centered Gaussian law (on $\R$).
By Remark \ref{rem:Gproduct} we get that for every $\xi\in \partial X$, for every $g_1,\dots,g_n\in G$, 
$$d(g_n\dots g_1 o,o)- \sigma(g_n\dots g_1,\xi)= 2(g_1^{-1}\dots g_n^{-1} o|\xi)_o$$
Since $g_1^{-1}\dots g_n^{-1}o$ converges almost surely to some random point  whose distribution $\check \nu$ is non-atomic, and since $g_n\dots g_1$ has the same distribution as $Z_n=g_1\dots g_n$ (namely $\mu^{*n}$) the sequences $\left( \frac{d(Z_no,o)-n\lambda}{\sqrt n}\right)_{n\geq 1}$ and $\left(\frac{ \sigma(g_n\dots g_1,\xi)-n\lambda}{\sqrt n}\right)$ have the same limit distribution. The variance of this Gaussian law is given by $\int_{G\times \partial X} (\sigma_0(g,\xi))-\lambda)^2 d\mu(g) d\nu(\xi)$, which is precisely $\int_{G\times \partial X} (h_\xi(g^{-1}o)-\psi(\xi)+\psi(g\xi)-\lambda)^2 d\mu(g) d\nu(x)$, where $\psi$ is given by Lemma~\ref{lem:centerable}.
\end{proof}

\subsection{Non-degeneracy of the limit law}

Our goal in this section is to complete the proof of Theorem~\ref{maintheorem}, that is, to prove that the limit Gaussian law is non-degenerate.  By Theorem~\ref{thm:gaussian}, this amounts to proving the following.

\begin{proposition}\label{prop:non-degeneracy}
If the action of $G$ on $X$ is nonelementary and essential and $\<\Supp(\mu)\>^+=G$ then 
$$\int_{G\times \partial X} (h_\xi( g ^{-1}o)+\psi(\xi)-\psi(g \xi) -\lambda)^2 d\mu(g) d\nu(\xi)>0.$$
\end{proposition}

In order to prove this we need to use the following definition.

\begin{definition}
Let $g$ be an isometry of an essential CAT(0) cubical complex $X_1\times \dots\times X_D$, where each $X_i$ is irreducible.

We say that $g$ is \emph{regular} if it acts as a regular rank-1 isometry on $X_i$, for every $1\leq i \leq D$.
\end{definition}

By \cite[Theorem 11.7]{FLM}, under the assumptions of Proposition~\ref{prop:non-degeneracy}, the group $G$ always contains regular isometries.

\begin{lemma}\label{lem:regularcontracting}
Let $X=X_1\times \dots\times X_D$ be the irreducible decomposition for $X$ and $g$ be a regular isometry. Then $g$ has exactly $2^D$-fixed points $\fix(g) \subset \overline X$, with exactly one attracting $\xi^+\in \fix(g)$ and one repelling $\xi^-\in \fix(g)$. Specifically, for every $x\in \overline X\setminus\fix(g)$ we have $\Lim{n\to \pm\8}{}\,  g^n x = \xi^\pm$. Furthermore, $I(\xi^-,\xi^+)\cap X$ is the set of  vertices that lie on some combinatorial axis of $g$, and if $o\in I(\xi^-,\xi^+)\cap X$ then $m(g^{-1}o,o, \xi^+) = o$. 
\end{lemma}
\begin{proof}

Let $X=X_1\times \dots\times X_D$, where each $X_i$ is irreducible. Let  $g=(g_1, \dots, g_D)$, where $g_i$ is regular rank-1. Using again \cite[Proposition 4.3]{BeyFio}  there are exactly 2 fixed points of $g_i$ in $\overline X_i$, call them $\xi_i^+$ and $\xi_i^-$, and furthermore $g_i^nx$ converges to $\xi_i^+$ for every $x_i\in\~X_i\setminus\{\xi_i^-\}$. Thus $\fix(g)$ is the set of $2^D$ elements whose coordinates are all the $\xi_i^\pm$.  Furthermore the above argument applied in each factor shows that $g_{\pm n} x\to (\xi_1^{\pm},\dots,\xi_D^{\pm})$ for every $x\not\in \fix(g)$. Let $\xi^+ = (\xi_1^+,\dots,\xi_D^+)$ and $\xi^- = (  \xi_1^-,\dots,\xi_D^-)$.

We now show if $I(\xi^-, \xi^+)\cap X$  is the set of  vertices that lie on some combinatorial axis of $g$.   By \cite[Proposition 4.4]{BeyFio} we have that $o \in I(\xi^-, \xi^+)\cap X$  if and only if $d(o, go) = \ell(g)$. Since $g^n$ is also  regular rank-1 with the same fixed points as $g$, it follows that $d(o, g^no) = \ell(g^n)= |n|\ell(g)$, for every $n\in \Z$ i.e. if $n>1$ then $d(o,go)+ d(go,g^no) = d(o, g^no)$ and $m(o,go,g^no) = go$. Therefore, $o\in I(\xi^-, \xi^+)\cap X$ if and only if $o$ belongs to a combinatorial axis for $g$. Taking the limit as $n\to \8$ we get that $m(o,go,\xi^+) = go$. Finally, applying $g^{-1}$ we deduce that  $m(g^{-1}o,o,\xi^+) = o$.\end{proof}

\begin{lemma}\label{Lem: horo to translation length}
Let $o\in X$ be any basepoint. If $g$ is a regular isometry, with attracting fixed point $\xi^+$, then $\lim\limits_n \frac{h_{\xi^+}(g^{-n} o)} n= \ell(g)$.
\end{lemma}

\begin{proof}
First, assume that $o\in I(\xi^-, \xi^+)\cap X$. By Lemma \ref{lem:regularcontracting} we have that 
%
$m(o,g^{-1}o,\xi^+) = o$, so that $h_{\xi^+}(g^{-1} o) = d(o,g^{-1}o)=\ell(g)$. Similarly we have $h_{\xi^+}(g^{-n} o) = n\ell(g)$ for every $n\geq 0$.

If $o\notin I(\xi^-, \xi^+)$ choose $o'\in I(\xi^-, \xi^+)$  and let $h'$ be the horofunction based at $o'$. Then we have $|h_{\xi^+}(y)-h'_{\xi^+}(y)|\leq d(o,o')$ for every $y\in X$. Since $h_{\xi^+}(g^{-n} o) = n\ell(g)$ we get  $|\frac{h'_{\xi^+}(g^{-n}o)} n-\ell(g)|\leq \frac{d(o,o')} n$. Taking the limit we get the result.
\end{proof}

\begin{proof}[Proof of Proposition~\ref{prop:non-degeneracy}]
By contradiction, assume that 
$$\int_{G\times \partial X} (h_\xi( g^{-1} o)+\psi(\xi)-\psi(g \xi) -\lambda)^2 d\mu(g) d\nu(\xi)=0$$

 This implies that for almost every $\xi \in \Supp(\nu)$ and $g\in \Supp(\mu)$ we have $h_{\xi}(g^{-1}o) - \lambda=\psi(g\xi)- \psi(\xi)$. Since $\psi$ is uniformly bounded, we get that  $|h_{\xi}(g^{-1}o) - \lambda|\leq C$ for some $C$. Fixing $g\in \Supp(\mu)$, the set of $\xi$ satisfying this inequality is closed and of full $\nu$-measure. As the support of $\nu$ is, by definition, the intersection of all closed full subsets of full measure, 
we get that this inequality holds for every $\xi\in \Supp(\nu)$. Since $\Supp(\mu)$ is countable, this inequality holds for all $\xi \in \Supp(\nu)$ and $g\in \Supp(\mu)$.

Now fix $n>0$. The random walk of law $\mu^{\ast n}$ has finite second moment, the same stationary measure as $\mu$, but has drift $n\lambda$ instead. Furthermore one can see a random walk of law $\mu^{\ast n}$ as a subsequence of a random walk of law $\mu$.  Applying Theorem~\ref{thm:gaussian}, we get that
$|h_{\xi}(g^{-1}o) - n\lambda|\leq C$ for every $\xi \in \Supp(\nu)$ and every $g\in \Supp(\mu^{\ast n})$, for every $n>0$.

Now let $g\in G$ be a regular isometry. Let $\xi^+$ (respectively $\xi^-$) be its attracting (respectively repelling) fixed point. Since the action of $G$ on $X$ is nonelementary, in particular, it is Roller nonelementary, that is, it has no finite orbits in $\~X$ (see Remark \ref{rem: basic action defs}) and so the stationary measure $\nu$ is non-atomic. Therefore by Lemma~\ref{lem:regularcontracting} for $\nu$-almost every point $\xi\in \partial X$ we have that $g^n \xi\to \xi^+$. This implies that $\xi^+\in \Supp (\nu)$.

By assumption, there exists $m>0$ such that $g\in \Supp(\mu^{\ast m})$. Hence for every $n>0$ we have $g^n\in \Supp(\mu^{\ast mn})$ and therefore $|h_{\xi^+}(g^{-n}o) - nm\lambda|\leq C$. Dividing by $n$ and taking the limit we get that $\ell(g)=m\lambda$ by Lemma \ref{Lem: horo to translation length}. 

Now, there exists $N>0$ such that $1\in \Supp(\mu^{\ast N})$. Therefore we also have $g\in \Supp(\mu^{\ast (m+N)})$. Hence the previous argument also proves that $\ell(g)=(m+N)\lambda$. Since $\lambda>0$ this is a contradiction.
\end{proof}

We can now conclude the proof of our main theorem.

\begin{proof}[Proof of Theorem~\ref{maintheorem}]
First, note that if we replace $o$ by some other vertex $o'\in X$ then $d(Z_no,o)-d(Z_no',o')$ is bounded and therefore the limit of $\frac{d(Z_n o,o)-n\lambda}{\sqrt n}$ does not depend on $o$. If the action of $G$ on $X$ were not essential, since the action is nonelementary then by \cite[Proposition 3.5]{CapraceSageev} there exists some $G$-invariant cubical subcomplex $X'\subset X$ on which the action of $G$ on $X'$ is essential and still nonelementary. Therefore choosing $o\in X'$ it suffices to prove the theorem for $X'$. In other words, we may and shall assume that the action of $G$ on $X$ is also essential.

Let us write $X=X_1\times \cdots \times X_m$ where each $X_i$ is irreducible and $\Aut(X)$ contains $\Aut(X_1)\times \cdots \times\Aut(X_m)$ as a finite index subgroup (Theorem \ref{Th: CapraceSageev Fundamentals}). Then $G$ has a finite index subgroup $G_0$ which preserves each factor. Since $G_0$ has finite index in $G$, it is recurrent, so that we can define the probability measure $\mu_0$ on $G_0$ as the law of the first time the random walk (starting at the identity) goes back to an element of $G_0$.  Using \cite[Lemma 2.3]{Kaima91} we see that $\mu_0$ has  finite second moment. Therefore, by Theorem~\ref{thm:gaussian}, the random walk of law $\mu_0$ satisfies a Central Limit Theorem. More precisely, if $(Y_n)$ is such a random walk on $G_0$, then the sequence $\left( \frac{d(Y_no,o)-n\lambda_0}{\sqrt n}\right)_{n\geq 1}$ converges to a Gaussian law of variance say $\Phi_0$.  By Proposition~\ref{prop:non-degeneracy} we also have $\Phi_0>0$.

Consider a coset $gG_0$ in $G/G_0$. Let $(Z_{\phi(n)})$ be the subsequence of $(Z_n)$ formed by the elements which are in $gG_0$. Then for every $n\geq 1$, $Z_{\phi(n)}$ has the same law as $g'Y_n$, where $g'$ is a random element (the first element of the random walk in the class $gG_0$), and $Y_n$ is the random walk on $G_0$ of law $\mu_0$. Furthermore, using the recurrence of $G_0$ in $G$ we have  that there exists a $C>0$ such that $\frac{\phi(n)}{n}$ converges to $C$, and it follows that $\lambda_0=C\lambda$ (see for example \cite[Lemma 9.7]{FLM}).

By the triangle inequality we have $d(Z_{\phi(n)}o,o)\leq d(Y_no,o)+d(g' o,o)$, and $\frac{d(g'o,o)}{\sqrt n}$ converges to $0$ in probability (since the law of $g'$ does not depend on $n$). So $\left( \frac{d(Z_{\phi(n)}o,o)-n\lambda_0}{\sqrt n}\right)_{n\geq 1}$ converges to a Gaussian law of variance $\Phi_0$. Hence $\left(\frac{d(Z_{\phi(n)} o,o) - \phi(n) \lambda}{\sqrt {\phi(n)}}\right)_{n\geq 1}$ converges to a Gaussian law of variance $\frac{\Phi_0}{\sqrt C}$.
Since this is valid for every class of $G/G_0$, we deduce that $\left( \frac{d(Y_no,o)-n\lambda}{\sqrt n}\right)_{n\geq 1}$ also converges to a Gaussian law of variance $\frac{\Phi_0}{\sqrt C}>0$.
\end{proof}

\bibliographystyle{alpha}
\bibliography{biblio}

\end{document}